\documentclass[12pt,reqno]{amsart}

\usepackage{amssymb,bbm,units,url}
\usepackage{booktabs}
\usepackage[bookmarks=false,unicode,colorlinks,urlcolor=red,citecolor=red,linkcolor=blue]{hyperref}
\usepackage{aliascnt}
\usepackage{doi}
\usepackage{tikz}
\usepackage{pgfplots} 

\usepackage{color}

\usepackage[margin=1.25in]{geometry}

\newtheorem{theorem}{Theorem}

\newaliascnt{lemma}{theorem}
\newtheorem{lemma}[lemma]{Lemma}
\aliascntresetthe{lemma}

\newaliascnt{proposition}{theorem}
\newtheorem{proposition}[proposition]{Proposition}
\aliascntresetthe{proposition}

\newaliascnt{corollary}{theorem}
\newtheorem{corollary}[corollary]{Corollary}
\aliascntresetthe{corollary}

\newaliascnt{conjecture}{theorem}
\newtheorem{conjecture}[conjecture]{Conjecture}
\aliascntresetthe{conjecture}

\newaliascnt{example}{theorem}
\newtheorem{example}[example]{Example}
\aliascntresetthe{example}

\makeatletter
\def\tagform@#1{\maketag@@@{\ignorespaces#1\unskip\@@italiccorr}}
\let\orgtheequation\theequation
\def\theequation{(\orgtheequation)}
\makeatother
\def\equationautorefname~{}

\theoremstyle{definition}
\newtheorem*{definition}{Definition}
\newtheorem*{remark}{Remark}

\newcommand{\arxiv}[1]{%
 \href{http://front.math.ucdavis.edu/#1}{ArXiv:#1}}

\DeclareMathOperator*{\argmax}{\arg\!\max}
\DeclareMathOperator*{\argmin}{\arg\!\min}

\newcommand{\e}{\varepsilon}

\newcommand{\N}{{\mathbb N}}
\newcommand{\R}{{\mathbb R}}

\newcommand{\area}{\operatorname{Area}}
\newcommand{\ud}{\,\mathrm{d}}
\begin{document}

\title[Optimal stretching for lattice points]{Optimal stretching for lattice points and eigenvalues}
\author[]{Richard S. Laugesen and Shiya Liu}
\address{Department of Mathematics, University of Illinois, Urbana,
IL 61801, U.S.A.}
\email{\ Laugesen\@@illinois.edu \; sliu63\@@illinois.edu (shiyaliu.zju\@@gmail.com)}
\date{\today}


\begin{abstract}
We aim to maximize the number of first-quadrant lattice points in a convex domain with respect to reciprocal stretching in the coordinate directions. The optimal domain is shown to be asymptotically balanced, meaning that the stretch factor approaches $1$ as the ``radius" approaches infinity. In particular, the result implies that among all $p$-ellipses (or Lam\'{e} curves), the $p$-circle encloses the most first-quadrant lattice points as the radius approaches infinity, for $1<p< \infty$. 

The case $p=2$ corresponds to minimization of high eigenvalues of the Dirichlet Laplacian on rectangles, and so our work generalizes a result of Antunes and Freitas. Similarly, we generalize a Neumann eigenvalue maximization result of van den Berg, Bucur and Gittins. Further, Ariturk and Laugesen recently handled $0<p<1$ by building on our results here.

The case $p=1$ remains open, and is closely related to minimizing energy levels of harmonic oscillators: which right triangles in the first quadrant with two sides along the axes will enclose the most lattice points, as the area tends to infinity? 
\end{abstract}

\maketitle

\vspace*{-12pt}

\section{\bf Introduction}

Among ellipses of given area centered at the origin and symmetric about both axes, which one encloses the most integer lattice points in the open first quadrant? One might guess the optimal ellipse would be circular, but a non-circular ellipse can enclose more lattice points, as shown in \autoref{fig:counterexample}. Nonetheless, optimal ellipses must become more and more circular as the area increases to infinity, by a striking result of Antunes and Freitas \cite{AF13}. 

To formulate the problem more precisely, consider the number of positive-integer lattice points lying in the elliptical region
\[
\Big( \frac{x}{s^{-1}} \Big)^{\! 2} + \Big( \frac{y}{s} \Big)^{\! 2} \leq r^2 ,
\]
where the ellipse has ``radius'' $r>0$ and semiaxes proportional to $s^{-1}$ and $s$. Notice that the area $\pi r^2$ of the ellipse is independent of the ``stretch factor'' $s$. Denote by $s=s(r)$ a value (not necessarily unique) of the stretch factor that maximizes the lattice point count. The theorem of Antunes and Freitas says $s(r) \to 1$ as $r \to \infty$, as illustrated in \autoref{fig:p2}. In other words, optimal ellipses become circular in the infinite limit. (Their theorem was stated differently, in terms of minimizing the $n$-th eigenvalue of the Dirichlet Laplacian on rectangles, with the square being asymptotically minimal. \autoref{sec:relation} explains the connection.) The analogous result for optimal ellipsoids becoming asymptotically spherical was proved recently in three dimensions by van den Berg and Gittins \cite{BBG16b} and in higher dimensions by Gittins and Larson \cite{GL17}, once again in the eigenvalue formulation. 

\begin{figure}
\includegraphics[scale=0.35]{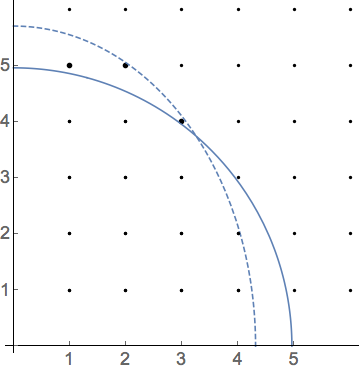}
\caption{\label{fig:counterexample}Circle $s=1$ and ellipse $s=1.15$, for radius $r=4.96$. The ellipse encloses three more points than the circle, as shown in bold.}
\end{figure}

This paper extends the result of Antunes and Freitas from circles to essentially arbitrary concave curves in the first quadrant that decrease between the intercept points $(0,1)$ and $(1,0)$. The ``ellipses'' in this situation are the images of the concave curve under rescaling by $s^{-1}$ and $s$ in the horizontal and vertical directions, respectively. \autoref{thm:s_bounded} says the maximizing $s(r)$ is bounded. \autoref{th:S_limit} shows under a mild monotonicity hypothesis on the second derivative of the curve that $s(r) \to 1$ as $r \to \infty$. Thus the most ``balanced'' curve in the family will enclose the most lattice points in the limit. 

Marshall \cite{marshall} recently extended this result to higher dimensions by somewhat different methods. We have generalized also to translated lattices in $2$-dimensions \cite{Lau_Liu2}.

\begin{figure}
\includegraphics[scale=0.45]{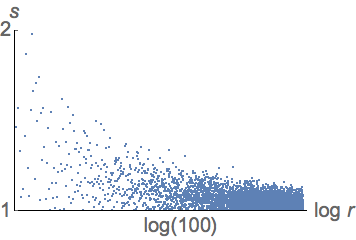}
\caption{\label{fig:p2}Optimal $s$-values for maximizing the number of lattice points in the $2$-ellipse. The graph plots the largest value $s(r)$ versus $\log r$. The plotted $r$-values are multiples of $\sqrt{3}/10$, an irrational number chosen in the hope of exhibiting generic behavior. The horizontal axis is at height $s=1$.}
\end{figure}

\autoref{th:S_limit_general} allows the curvature to blow up or degenerate at the intercept points, which permits us to treat the family of $p$-ellipses for $1<p<\infty$. In each case the $p$-circle is asymptotically optimal for the lattice counting problem in the first quadrant. The case $p=1$ is open problem. Our numerical evidence in \autoref{sec:oneellipse} suggests that the first-quadrant right triangle enclosing the most lattice points does \emph{not} necessarily approach a 45--45--90 degree triangle as $r \to \infty$. Instead one seems to get an infinite limit set of optimal triangles. See \autoref{conj:p_1}, where we describe the recent proof by Marshall and Steinerberger \cite{marshall_steinerberger}.

If one counts lattice points in the \emph{closed} first quadrant, that is, counting points on the axes as well, then the results reverse direction from maximization to minimization of the lattice count. \autoref{th:S_limit_neumann} shows that the value $s=s(r)$ minimizing the number of enclosed lattice points will tend to $1$ as $r \to \infty$. In the case of circles and ellipses, this result was obtained recently by van den Berg, Bucur and Gittins \cite{BBG16a} (and in higher dimensions by Gittins and Larson \cite{GL17}, generalized by Marshall \cite{marshall}). As explained in \autoref{sec:relation}, they showed that the maximizing rectangle for the $n$-th eigenvalue of the Neumann Laplacian must approach a square as $n \to \infty$. 

This paper builds on the framework of Antunes and Freitas for ellipses,  with new ingredients introduced to handle general concave curves. First we develop a new non-sharp bound on the counting function  (\autoref{prop:counting_ineq}) in order to control the stretch factor $s(r)$. Then we prove more precise lattice counting estimates (\autoref{th:asymptotic}) of Kr\"{a}tzel type, relying on a theorem of van der Corput (\autoref{app-exponential}). 

Convex decreasing curves in the first quadrant, such as $p$-ellipses with $0<p<1$, have been treated by Ariturk and Laugesen \cite{AL17} by building on this paper's results. 

\subsection*{Spectral motivations and results} 
This paper is inspired by recent efforts to understand the behavior of high eigenvalues of the Laplacian. Write $\lambda_n$ for the $n$-th eigenvalue of the Dirichlet Laplacian on a bounded domain $\Omega$ of area $1$ in the plane. (We restrict to $2$ dimensions for simplicity.) Denote the minimum value of this eigenvalue over all such domains by $\lambda_n^*$, and suppose it is achieved on a domain $\Omega_n^*$. What can one say about the shape of this minimizing domain? 

For the first eigenvalue, the minimizing domain $\Omega_1^*$ is a disk, by the Faber--Krahn inequality. For the second eigenvalue, $\Omega_2^*$ is a union of two disjoint equal-area disks, as Krahn and later P. Szego showed. A long-standing conjecture says $\Omega_3^*$ should be a disk and $\Omega_4^*$ should be a union of disjoint non-equal-area disks. For higher eigenvalues ($n \geq 5$),  minimizing domains found numerically do not have recognizable shapes; see \cite{AF12,Oud04} and references therein. Antunes and Freitas remark, though, that the ``most natural guess'' is $\Omega_n^*$ approaches a disk as $n \to \infty$, which is known to occur if the area normalization is strengthened to a perimeter normalization \cite{AF16,BF13}. This conjecture would imply the famous P\'olya conjecture $\lambda_n \geq 4\pi n /|\Omega|$, as Colbois and El Soufi \cite[Corollary~2.2]{colbois_soufi} showed using subadditivity of $n \mapsto \lambda_n^*$.

A partial result of Freitas \cite{f16} succeeds in determining the leading order asymptotic as $n \to \infty$ of the minimum value of the eigenvalue sum $\lambda_1+\dots+\lambda_n$ (rather than of $\lambda_n$ itself). This result provides no information on the shapes of the minimizing domains. Larson \cite{Larson} shows among convex domains that the disk asymptotically maximizes the Riesz means of the Laplace eigenvalues, for Riesz exponents $\geq 3/2$. If this Riesz exponent could be lowered to $0$, giving asymptotic maximality of the disk for the eigenvalue counting function, then one would obtain the desired conjecture about the eigenvalue minimizing domain $\Omega_n^*$.

A complete resolution for rectangular domains was found by Antunes and Freitas \cite{AF13}, using lattice counting methods as explained in \autoref{sec:relation}. They proved that the minimizing domain for $\lambda_n$ among rectangles approaches a square as $n \to \infty$. Similarly, the cube is asymptotically minimal in $3$ and higher dimensions \cite{BBG16b,GL17}.

\vspace*{-5pt}

\subsection*{Open problem for the harmonic oscillator}
Asymptotic optimality of the square for minimizing Dirichlet eigenvalues of the Laplacian on rectangles suggests an analogous open problem for harmonic oscillators. Consider the Schr\"odinger operator in $2$ dimensions with parabolic potential $(sx)^2 + (y/s)^2$, where $s> 0$ is a parameter. Write $s(n)$ for a parameter value that minimizes the $n$-th eigenvalue of this operator. What is the limiting behavior of $s(n)$ as $n \to \infty$?

The results on rectangular domains (which may be regarded as infinite potential wells) might suggest $s(n) \to 1$, but we think that it is not the case. Instead we believe $s(n)$ might cluster around infinitely many values as $n \to \infty$. Indeed, after rescaling, the Schr\"odinger operator has eigenvalues of the form $s(j-1/2)+(k-1/2)/s$, which leads to a lattice point counting problem inside right triangles, like in \autoref{sec:oneellipse} for $p = 1$, except now the lattices are shifted by $1/2$ to the left and downwards. For the unshifted lattice, numerical work in \autoref{sec:oneellipse} suggests that the optimal stretching parameter $s$ does not converge to $1$, and instead has many cluster points as $r \to \infty$. Recent investigations \cite[Section~10]{Lau_Liu2} suggest that this clustering phenomenon persists for shifted lattices and hence for the harmonic oscillator.

\section{\bf Assumptions and definitions}
\label{sec:assumptions}

The first quadrant is the \emph{open} set $\{ (x,y) : x,y>0 \}$.

Assume throughout the paper that $\Gamma$ is a concave, strictly decreasing curve in the first quadrant. Our theorems assume the $x$- and $y$-intercepts of the curve are equal, occurring at $x=L$ and $y=L$ respectively, as shown in \autoref{fig:gammafig}. Write $\area(\Gamma)$ for the area enclosed by the curve $\Gamma$ and the $x$- and $y$-axes. 

\begin{figure}
\includegraphics[scale=0.35]{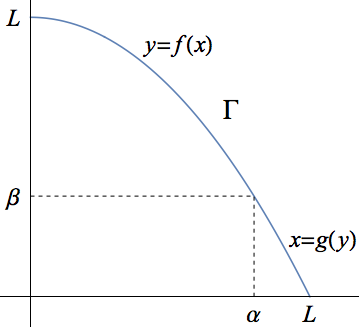}
\caption{\label{fig:gammafig}A concave decreasing curve $\Gamma$ in the first quadrant, with intercepts at $L$. The point $(\alpha,\beta)$ on the curve is relevant to \autoref{th:S_limit}.}
\end{figure}

We represent the curve $\Gamma$ by $y=f(x)$ for $x\in [0,L]$, so that $f$ is a concave strictly decreasing function, and of course $f$ is continuous. In particular 
\[
L=f(0)>f(x)>f(L)=0 
\]
whenever $x\in (0,L)$. Denote the inverse function of $f(x)$ by $g(y)$ for $y \in[0,L]$, so that $g$ also is concave and strictly decreasing.

We define a rescaling of the curve by parameter $r>0$:
\begin{align*}
r\Gamma 
& = \text{image of $\Gamma$ under the radial scaling $\left(\begin{smallmatrix}
r & 0 \\
0 & r
\end{smallmatrix}\right)$} \\
& = \text{graph of $rf({x/r})$,}
\end{align*}
and define an area-preserving stretch of the curve by:
\begin{align*}
\Gamma(s)
& = \text{image of $\Gamma$ under the diagonal scaling $\left(\begin{smallmatrix}
s^{-1} & 0\\
0 & s
\end{smallmatrix}\right)$} \\
& = \text{graph of $s  f({sx})$,}
\end{align*}
where $s>0$. In other words, $\Gamma(s)$ is obtained from $\Gamma$ after compressing the $x$-direction by $s$ and stretching the $y$-direction by $s$. Define the counting function for 
$r\Gamma(s)$ by 
\begin{align*}
N(r,s) &=\text{number of positive-integer lattice points lying inside or on $r\Gamma(s)$ } \\
&=\# \big\{ (j,k)\in\mathbb{N} \times \mathbb{N}:k\leq rsf(js/r) \big\}. 
\end{align*}
For each $r>0$, consider the set
\[
S(r) = \argmax_{s >0} N(r,s) 
\]
consisting of the $s$-values that maximize the number of first-quadrant lattice points enclosed by the curve $r\Gamma(s)$. The set $S(r)$ is well-defined because the maximum is indeed attained, as the following argument shows. The curve $r\Gamma(s)$ has $x$-intercept at $rs^{-1}L$, which is less than $1$ if $s>rL$ and so in that case the curve encloses no positive-integer lattice points. Similarly if $s<(rL)^{-1}$, then $r\Gamma(s)$ has height less than $1$ and contains no lattice points in the first quadrant. Thus for each fixed $r>0$, if $s$ is sufficiently small or sufficiently large then the counting function $N(r,s)$ equals zero, while obviously for intermediate values of $s$ the integer-valued function $s \mapsto N(r,s)$ is bounded. Hence $N(r,s)$ attains its maximum at some $s>0$.

For later reference, we write down this bound on optimal $s$-values. 
\begin{lemma}[$r$-dependent bound on optimal stretch factors]\label{lemma:bound}
If $\Gamma$ is a concave, strictly decreasing curve in the first quadrant with equal intercepts (as in \autoref{fig:gammafig}), then 
\[
S(r)\subset \big[(rL)^{-1}, rL\big] \qquad \text{whenever $r \geq 2/L$.}
\]
\end{lemma}
\begin{proof}
The curve $r\Gamma(1)$ has horizontal and vertical intercepts at $rL\geq 2$. Hence by concavity, $r\Gamma(1)$ encloses the  point $(1,1)$, and so the counting function $s \mapsto N(r,s)$ is greater than zero when $s=1$. On the other hand when $s< (rL)^{-1}$ or $s> rL$, we know $N(r,s)=0$ by the paragraph before the lemma. Thus the maximum can only be attained when $s$ lies in the interval $\big[(rL)^{-1}, rL\big]$. 
\end{proof}

\section{\bf Main results}
\label{sec:mainresults}

The curve $\Gamma$ has $x$- and $y$-intercepts at $L$, in the theorems in that follow; see \autoref{fig:gammafig}. We start by improving \autoref{lemma:bound} to show the maximizing set $S(r)$ is bounded, and the bounds can be evaluated explicitly in the limit as $r \to \infty$. 
\begin{theorem}[Uniform bound on optimal stretch factors] \label{thm:s_bounded}
If $\Gamma$ is a concave, strictly decreasing curve in the first quadrant then 
\[
S(r) \subset [s_1,s_2] \qquad \text{for all $r \geq 2/L$,}
\]
for some constants $s_1,s_2>0$. Furthermore, given $\e>0$,  
\[
S(r) \subset \big[\frac{1}{4+\e},4+\e\big] \qquad \text{for all large $r$.}
\]
\end{theorem}
The proof appears in \autoref{sec:boundedness}. 

If the concave decreasing curve is smooth with monotonic second derivative, then in addition to being bounded above and below the maximizing set $S(r)$ converges to $\{ 1 \}$, as the next theorem shows. Recall that $g$ is the inverse function of $f$. 
\begin{theorem}[Optimal concave curve is asymptotically balanced]\label{th:S_limit}
Assume $(\alpha,\beta) \in \Gamma$ is a point in the first quadrant such that $f \in C^2[0, \alpha]$ with $f'<0$ on $(0,\alpha]$ and $f'' < 0$ on $[0, \alpha]$, and similarly $g \in C^2[0, \beta]$ with $g'<0$ on $(0,\beta]$ and $g'' < 0$ on $[0, \beta]$. Further suppose $f''$ is monotonic on $[0,\alpha]$ and $g''$ is monotonic on $[0,\beta]$. 

Then the optimal stretch factor for maximizing $N(r,s)$ approaches $1$ as $r$ tends to infinity, with
\[
S(r) \subset \big[ 1-O(r^{-1/6}),1+O(r^{-1/6}) \big] ,
\]
and the maximal lattice count has asymptotic formula 
\[
\max_{s > 0} N(r,s) = r^2\area (\Gamma)-rL + O(r^{2/3}) .
\]
\end{theorem}
The theorem is proved in \autoref{sec:mainproof}. Slight improvements to the decay rate $O(r^{-1/6})$ and the error term $O(r^{2/3})$ are possible, as explained after \autoref{th:asymptotic}.  

\subsection*{More general curves for lattice counting}
We want to weaken the smoothness and monotonicity assumptions in \autoref{th:S_limit}. We start with a definition of piecewise smoothness. 
\begin{definition}[$PC^2$]\ 

(i) We say a function $f$ is piecewise $C^2$-smooth on a half-open interval $(0,\alpha]$ if $f$ is continuous and a partition $0=\alpha_0<\alpha_1 < \dots < \alpha_l = \alpha$ exists such that $f \in C^2(0,\alpha_1]$ and $f \in C^2[\alpha_{i-1},\alpha_i]$ for $i=2,\dots,l$. Write $PC^2(0,\alpha]$ for the class of such functions. 

(ii) Write $f^\prime<0$ to mean that $f^\prime$ is negative on the subintervals $(0,\alpha_1]$ and $[\alpha_{i-1},\alpha_i]$ for $i=2,\dots,l$, with the derivative being taken in the one-sided senses at the partition points $\alpha_1,\dots,\alpha_l$. The meaning of $f^{\prime \prime}<0$ is analogous.

(iii) We label partition points using the same letter as for the right endpoint. In particular, the partition for $g \in PC^2(0,\beta]$ is $0=\beta_0<\dots<\beta_\ell=\beta$.
\end{definition}

For the next theorem, take a point $(\alpha ,\beta) \in \Gamma$ lying in the first quadrant and suppose we have numbers $a_1,a_2,b_1,b_2>0$ and positive valued functions $\delta(r)$ and $\epsilon(r)$ such that as $r \to \infty$: 
\begin{align}
\delta(r) = O(r^{-2a_1}) , \qquad & f''\big(\delta(r)\big)^{-1} = O(r^{1-4a_2}) , \label{eq:f_sup}\\
\epsilon(r) = O(r^{-2b_1}) , \qquad & g''\big(\epsilon(r)\big)^{-1} = O(r^{1-4b_2})  . \label{eq:g_sup}
\end{align}
(The second condition in \autoref{eq:f_sup} says that $f''(x)$ cannot be too small as $x \to 0$.)  
Let
\[
e=\min \{ \tfrac{1}{6},a_1,a_2,b_1,b_2 \} .
\]
Now we extend \autoref{th:S_limit} to a larger class of concave decreasing curves. 
\begin{theorem}[Optimal concave curve is asymptotically balanced]\label{th:S_limit_general}\ 

\noindent Assume $f\in PC^2(0,\alpha]$ with $f'<0$ and $f''<0$, and $f^{\prime \prime}$ is monotonic on each subinterval of the partition. Similarly assume $g\in PC^2(0,\beta]$ with $g'<0$ and $g''<0$, and $g^{\prime \prime}$ is monotonic on each subinterval of the partition. Suppose the positive functions $\delta(r)$ and $\epsilon(r)$ satisfy conditions \autoref{eq:f_sup} and \autoref{eq:g_sup}. 

Then the optimal stretch factor for maximizing $N(r,s)$ approaches $1$ as $r$ tends to infinity, with
\[
S(r) \subset \big[ 1-O(r^{-e}),1+O(r^{-e}) \big] ,
\]
and the maximal lattice count has asymptotic formula 
\[
\max_{s > 0} N(r,s) = r^2\area (\Gamma)-rL + O(r^{1-2e}) .
\]
\end{theorem}
The proof is presented in \autoref{sec:generalproof}. 
\begin{example}[Optimal $p$-ellipses for lattice point counting]\rm \label{ex:p-ellipse}
Fix $1<p<\infty$, and consider the $p$-circle
\[
\Gamma : |x|^p+|y|^p=1 ,
\]
which has intercept $L=1$. That is, the $p$-circle is the unit circle for the $\ell^p$-norm on the plane. Then the $p$-ellipse
\[
r\Gamma(s) : |sx|^p + |s^{-1}y|^p \leq r^p
\]
has first-quadrant counting function 
\[
N(r,s) = \# \{ (j,k) \in \N \times \N :  (js)^p+(ks^{-1})^p \leq r^p \} .
\]

We will show that the $p$-ellipse containing the maximum number of positive-integer lattice points must approach a $p$-circle in the limit as $r \to \infty$, with 
\[
S(r) \subset [1-O(r^{-e}),1+O(r^{-e})] 
\]
where $e=\min \{ \tfrac{1}{6},\tfrac{1}{2p} \}$.

\autoref{th:S_limit} fails to apply to $p$-ellipses when $1<p <2$, because the second derivative of the curve  is not monotonic (see $f^{\prime \prime}(x)$ below), and the theorem fails to apply when $2<p<\infty$ because $f^{\prime \prime}(0)=0$ in that case. Instead we will apply \autoref{th:S_limit_general}. 

To verify that the $p$-circle satisfies the hypotheses of \autoref{th:S_limit_general}, we let $\alpha=\beta=2^{-1/p}$ and choose 
\[
\delta(r)=r^{-1/p} , \qquad \epsilon(r)=r^{-1/p} ,
\]
for all large $r$. Then $\delta(r)=r^{-2a_1}$ with $a_1=1/2p$. Next, 
\begin{align*}
f(x) & = (1-x^p)^{1/p} , \\
f'(x) & = -x^{p-1} (1-x^p)^{-1+1/p} , \\
f''(x) & =-(p-1)x^{p-2}(1-x^p)^{-2+1/p} ,
\end{align*}
so that
\[
\big| f''\big(\delta(r)\big) \big|^{-1} \leq (\text{const.}) r^{1-2/p} ,
\]
and hence $a_2=1/2p$ in \autoref{eq:f_sup}. Thus $f$ satisfies hypothesis \autoref{eq:f_sup}. 

Further, the interval $(0,\alpha)$ can be partitioned into subintervals on which $f''$ is monotonic, because the third derivative
\[
f'''(x) = -(p-1) x^{p-3} (1 - x^p)^{-3 + 1/p} \big( (1+p)x^p + p-2 \big)
\]
vanishes at most once in the unit interval. 

The calculations are the same for $g$, and so the desired conclusion for $p$-ellipses follows from \autoref{th:S_limit_general} when $1<p<\infty$. 

For $p=\infty$, the $\infty$-circle is a Euclidean square and the $\infty$-ellipse is a rectangle. Many different rectangles of given area can contain the same number of lattice points. For example, a $4 \times 1$ rectangle and $2 \times 2$ square each contain $4$ lattice points in the first quadrant. All such matters can be handled by the explicit formula $N(r,s)=\lfloor rs^{-1} \rfloor \lfloor rs \rfloor$ for the counting function when $p=\infty$.

The case $p=1$ is an open problem, as discussed in \autoref{sec:oneellipse}. The case $0<p<1$ has been handled by Ariturk and Laugesen \cite{AL17} using results in this paper.

Incidentally, an explicit estimate on the number of lattice points in the full $p$-ellipse in all four quadrants was obtained by Kr\"{a}tzel \cite[Theorem 2]{kratzel04} for $p \geq 2$. See the informative survey by Ivi\'c \emph{et al.}\ \cite[{\S}3.1]{IKKN06}.
\end{example}

\subsection*{Lattice points in the closed first quadrant, and Neumann eigenvalues} Our results have analogues for lattice point counting in the closed (rather than open) first quadrant, as we now explain. When counting nonnegative-integer lattice points, which means we include lattice points on the axes, the counting function for $r\Gamma(s)$ is
\begin{align*}
\mathcal{N}(r,s) =\#\{(j,k)\in\mathbb{Z}_{+} \times \mathbb{Z}_{+}:k\leq rsf(js/r)\} ,
\end{align*}
where $\mathbb{Z}_+ = \{0,1,2,3,\ldots\}$. Define 
\[
\mathcal{S}(r) = \argmin_{s >0} \mathcal{N}(r,s) . 
\]
In other words, the set $\mathcal{S}(r)$ consists of the $s$-values that minimize the number of lattice points inside the curve $r\Gamma(s)$ in the \emph{closed} first quadrant. Notice we employ the calligraphic letters $\mathcal N$ and $\mathcal S$ when working with nonnegative-integer lattice points. 

\begin{theorem}[Uniform bound on optimal stretch factors] \label{thm:s_bounded_neumann}
If $\Gamma$ is a concave, strictly decreasing curve in the first quadrant then 
\[
\mathcal{S}(r) \subset [s_1,s_2] \qquad \text{for all $r \geq 2/L$,}
\]
for some constants $s_1,s_2>0$.
\end{theorem}

\begin{theorem}[Optimal concave curve is asymptotically balanced]\label{th:S_limit_neumann}
Under the assumptions of \autoref{th:S_limit}, the optimal stretch factor for minimizing $\mathcal{N}(r,s)$ approaches $1$ as $r$ tends to infinity:
\begin{align*}
\mathcal{S}(r) & \subset [1-O(r^{-1/6}),1+O(r^{-1/6})] , \\
\min_{s > 0} \mathcal{N}(r,s) & = r^2\area (\Gamma) + rL + O(r^{2/3}) ,
\end{align*}
and under the assumptions of \autoref{th:S_limit_general} we have similarly that:
\begin{align*}
\mathcal{S}(r) & \subset [1-O(r^{-e}),1+O(r^{-e})] , \\
\min_{s > 0} \mathcal{N}(r,s) & = r^2\area (\Gamma) + rL + O(r^{1-2e}) .
\end{align*}
\end{theorem}
The proofs are in \autoref{sec:closedquadrant}. 

Consequently, we reprove a recent theorem of van den Berg, Bucur and Gittins \cite{BBG16a} saying that the optimal rectangle of area $1$ for maximizing the $n$-th Neumann eigenvalue of the Laplacian approaches a square as $n \to \infty$. See \autoref{sec:relation} for discussion, and an explanation of why the approach in this paper is simpler.

\section{\bf Proof of \autoref{thm:s_bounded}}
\label{sec:boundedness}
 
To control the stretch factors and hence prove \autoref{thm:s_bounded}, we will first derive a rough lower bound on the counting function, and then a more sophisticated upper bound. The leading order term in these bounds is simply the area inside the rescaled curve and thus is best possible, while the second term scales like the length of the curve and so at least has the correct order of magnitude. 

Assume $\Gamma$ is concave and decreasing in the first quadrant, with $x$- and $y$-intercepts at $L$ and $M$ respectively. The intercepts need not be equal, in the lemmas and proposition below. Recall that $N(r,s)$ counts the positive-integer lattice points under the curve $\Gamma$, while $\mathcal{N}(r,s)$ counts nonnegative-integer lattice points. 
\begin{lemma}[Relation between counting functions] \label{le:relation}
For each $r,s>0$, 
\[
\mathcal{N}(r,s)=N(r,s)+r(s^{-1}L+sM ) + \rho(r,s)
\]
for some number $\rho(r,s) \in [-1,1]$.
\end{lemma}
\begin{proof}
The difference between the two counting functions is simply the number of lattice points lying on the coordinate axes inside the intercepts of $r\Gamma(s)$. There are 
\[
\lfloor rs^{-1}L \rfloor + \lfloor rsM \rfloor +1
\]
such lattice points, and so the lemma follows immediately. 
\end{proof}
\begin{lemma}[Rough lower bound] \label{le:basic}
The number $N(r,s)$ of positive-integer lattice points lying inside $r\Gamma(s)$ in the first quadrant satisfies
\[
N(r,s) \geq r^2\area(\Gamma) - r(s^{-1}L+sM) - 1 , \qquad r,s > 0 .
\]
\end{lemma}
\begin{proof}
Notice $\mathcal{N}(r,s)$ equals the total area of the squares of sidelength $1$ having lower left vertices at nonnegative-integer lattice points inside the curve $r\Gamma(s)$. The union of these squares contains $r\Gamma(s)$, since the curve is decreasing. Hence $\mathcal{N}(r,s) \geq r^2\area(\Gamma)$, and so 
\begin{align*}
N(r,s)
& \geq \mathcal{N}(r,s) - r(s^{-1}L+sM ) - 1 \qquad \text{by \autoref{le:relation}} \\
& \geq r^2\area(\Gamma) - r(s^{-1}L+sM ) - 1 .
\end{align*}
\end{proof}

For the upper bound in the next proposition, remember $\Gamma$ is the graph of $y=f(x)$, where $f$ is concave and decreasing on $[0,L]$, with $f(0)=M, f(L)=0$. We do not assume $f$ is differentiable in the next result, although in order to guarantee the constant $C$ in the proposition is positive, we assume $f$ is \emph{strictly} decreasing. 
\begin{proposition}[Two-term upper bound on counting function]\label{prop:counting_ineq} 
Let $C=M-f(L/2)$.

\noindent (a) The number $N$ of positive-integer lattice points lying inside $\Gamma$ in the first quadrant satisfies
\begin{equation}
N \leq \area(\Gamma) - \frac{1}{2} C
\end{equation}
provided $L \geq 1$.

\noindent (b) The number $N(r,s)$ of positive-integer lattice points lying inside $r\Gamma(s)$ in the first quadrant satisfies
\[
N(r,s)\leq r^2\area(\Gamma) - \frac{1}{2}Crs
\]
whenever $r \geq s/L$.
\end{proposition}
\begin{proof}
\begin{figure}
\includegraphics[scale=0.4]{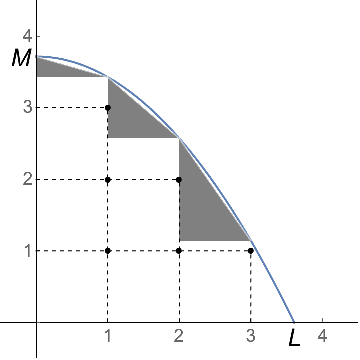}
\caption{\label{fig:dirichlet_counting}Positive integer lattice count satisfies $N \leq \area(\Gamma) - \area(\text{triangles})$, in proof of \autoref{prop:counting_ineq}(a).}
\end{figure}

Part (a). 
Clearly $N$ equals the total area of the squares of sidelength $1$ having upper right vertices at positive-integer lattice points inside the curve $\Gamma$. Consider also the right triangles of width $1$ formed by secant lines on $\Gamma$ (see \autoref{fig:dirichlet_counting}), that is, the triangles with vertices $\big( i-1,f(i-1) \big),\big( i,f(i) \big),\big( i-1,f(i) \big)$, where $i=1,\ldots,\lfloor L \rfloor$. These triangles lie above the squares by construction, and lie below $\Gamma$ by concavity. Hence 
\begin{equation}\label{eq:total_area}
N+\area(\text{triangles})\leq \area(\Gamma) .
\end{equation}
Since $f$ is decreasing, we find
\begin{align}
\area(\text{triangles})&=\sum_{i=1}^{\lfloor L\rfloor} \frac{1}{2}\big(f(i-1) - f(i)\big) \nonumber\\
&=\frac{1}{2}\big(f(0)-f(\lfloor L\rfloor)\big)\label{eq:triangle-area-exact}\\
&\geq \frac{1}{2}\big(M- f(L/2)\big) = \frac{1}{2} C, \label{eq:area}
\end{align}
because $\lfloor L \rfloor \geq L/2$ when $L \geq 1$. Combining \autoref{eq:total_area} and \autoref{eq:area}  proves Part (a).

\smallskip
Part (b). 
Simply replace $\Gamma$ in Part (a) with the curve $r\Gamma(s)$, meaning we replace $L, M, f(x)$ with $rs^{-1}L, rsM, rsf(sx/r)$ respectively.
\end{proof}

\subsection*{Proof of \autoref{thm:s_bounded}} Recall the intercepts are assumed equal ($L=M$) in this theorem. Let $r \geq 2/L$ and suppose $s \in S(r)$. Then $r \geq s/L$ by \autoref{lemma:bound}, and so the upper bound in \autoref{prop:counting_ineq}(b) gives
\[
N(r,s)\leq r^2\area(\Gamma) - \frac{1}{2}Crs .
\]
The lower bound in \autoref{le:basic} with ``$s=1$'' says 
\begin{equation}\label{eq:s-1-lower-bound}
N(r,1) \geq r^2\area(\Gamma) - 2rL - 1 .
\end{equation}

The value $s \in S(r)$ is a maximizing value, and so $N(r,1) \leq N(r,s)$. The preceding inequalities therefore imply
\[
\frac{1}{2}Crs \leq 2rL+1 \leq \frac{5}{2}rL . 
\]
Hence $s \leq 5L/C \equiv s_2$, and so the set $S(r)$ is bounded above.

Interchanging the roles of the horizontal and vertical axes, we similarly find $s^{-1} \leq 5L/\widetilde{C} \equiv s_1^{-1}$,  so that the set $S(r)$ is bounded below away from $0$, completing the first part of the proof. 

The fact that $S(r)$ is bounded will help imply an improved bound in the limit as $r \to \infty$. Going back to the proof of \autoref{prop:counting_ineq}(a), we see from \autoref{eq:total_area} and \autoref{eq:triangle-area-exact} that 
\[
N+\frac{1}{2}\big(f(0)-f(\lfloor L\rfloor) \big) \leq \area(\Gamma).
\]
Rescaling the curve from $\Gamma$ to $r\Gamma(s)$, so that $N$ and $f(x)$ become $N(r,s)$ and $rsf\big(\frac{s}{r}x\big)$, respectively, and the $x$-intercept $L$ becomes $rL/s$, we see the last inequality becomes
\[
N(r,s) + \frac{1}{2} rs\big(f(0)-f(\frac{s}{r}\lfloor \frac{rL}{s}\rfloor)\big)\leq r^2\area(\Gamma).
\]
Hence
\begin{equation*}\label{eq:p_optimal}
N(r,s) \leq r^2\area(\Gamma)-\frac{1}{2}rsL + o(r) \qquad \text{as $r \to \infty$,}
\end{equation*}
where to get the error term $o(r)$ we used that $s \in S(r)$ is bounded above and below ($s_1 \leq s \leq s_2$) and $f(L)=0$.
Since $s$ is a maxmizing value we have $N(r,1)\leq N(r,s)$, and so \autoref{eq:s-1-lower-bound} and the above inequality imply
\[
\frac{1}{2}rsL+o(r) \leq 2rL + 1 ,
\]
which implies $\limsup_{r \to \infty} s \leq 4$. Similarly $\limsup_{r\to \infty} s^{-1} \leq 4$, by interchanging the axes. 

\section{\bf Two-term counting estimates with explicit remainder}

We start with a result for $C^2$-smooth curves. What matters in the following proposition is that the right side of estimate \autoref{eq:gamma_a_asymptotic} below has the form $O(r^\theta)$ for some $\theta<1$, and that the $s$-dependence in the estimate can be seen explicitly. The detailed dependence on the functions $f$ and $g$ will not be important for our purposes. 

The horizontal and vertical intercepts $L$ and $M$ need not be equal, in this section. 
\begin{proposition}[Two-term counting estimate]\label{th:asymptotic}
Take a point $(\alpha ,\beta) \in \Gamma$ lying in the first quadrant, and assume that $f \in C^2[0, \alpha]$ with $f'<0$ on $(0,\alpha]$ and $f'' < 0$ on $[0, \alpha]$, and similarly $g \in C^2[0, \beta]$ with $g'<0$ on $(0,\beta]$ and $g'' < 0$ on $[0, \beta]$. Further suppose $f''$ is monotonic on $[0,\alpha]$ and $g''$ is monotonic on $[0,\beta]$. 

\noindent (a) The number $N$ of positive-integer lattice points inside $\Gamma$ in the first quadrant satisfies:
\begin{align*}
&\big|N-\area (\Gamma)+(L+M)/2\big|\nonumber\\
&\leq 6\Big(\int_0^\alpha |f''(x)|^{1/3} \ud x+\int_0^\beta |g''(y)|^{1/3}\ud y\Big)  +175\big(\max_{ [0,\alpha]}\frac{1}{|f''|^{1/2}}+\max_{ [0,\beta]}\frac{1}{|g''|^{1/2}}\big)\nonumber\\
& \hspace{4cm}+\frac{1}{4} \big( |f'(\alpha)|+|g'(\beta)| \big)+3 .
\end{align*}

\noindent (b) The number $N(r,s)$ of positive-integer lattice points lying inside $r\Gamma(s)$ in the first quadrant satisfies (for $r,s>0$):
\begin{align}
&\big|N(r,s)-r^2\area (\Gamma)+r(s^{-1}L+sM)/2\big|\nonumber\\
&\leq 6r^{2/3}\Big(\int_0^\alpha |f''(x)|^{1/3} \ud x+\int_0^\beta |g''(y)|^{1/3}\ud y\Big)  +175r^{1/2}\big(\max_{ [0,\alpha]}\frac{s^{-3/2}}{|f''|^{1/2}}+\max_{[0,\beta]}\frac{s^{3/2}}{|g''|^{1/2}}\big)\nonumber\\
& \hspace{4cm}+\frac{1}{4}(s^2|f'(\alpha)|+s^{-2}|g'(\beta)|)+3. \label{eq:gamma_a_asymptotic}
\end{align}
\end{proposition}
\autoref{th:asymptotic} and its proof are closely related to work of Kr\"{a}tzel \cite[Theorem~1]{kratzel04}. We give a direct proof below for two reasons: we want the estimate \autoref{eq:gamma_a_asymptotic} that depends explicitly on the stretching parameter $s$, and we want a proof that can be modified to use a weaker monotonicity hypothesis, in \autoref{th:asymptotic_general}. 

A better bound on the right side of \autoref{eq:gamma_a_asymptotic}, giving order $O(r^{\theta+\epsilon})$ with $\theta = 131/208 \simeq 0.63 < 2/3$, can be found in work of Huxley \cite{Hux03}, with precursors in \cite[Theorems~18.3.2 and 18.3.3]{Hux96}. That bound is difficult to prove, though, and the improvement is not important for our purposes since it leads to only a slight improvement in the rate of convergence for $S(r)$, namely from $O(r^{-1/6})$ to $O(r^{(\theta+\epsilon-1)/2})$ in \autoref{th:S_limit}. 
\begin{proof}
Part (a). 
We divide the region under $\Gamma$ into three parts. Let $N_1$ count the lattice points lying to the left of the line $x=\alpha$ and above $y=\beta$, and $N_2$ count the lattice points to the right of $x=\alpha$ and below $y=\beta$, and $N_3$ count the lattice points in the remaining rectangle $(0,\alpha] \times (0,\beta]$. That is, 
\begin{align*}
N_1&=\sum_{0<m\leq \alpha} \ \sum_{\beta <n\leq f(m)} 1 = \sum_{0<m\leq \alpha} \big( \lfloor f(m)\rfloor-\lfloor \beta \rfloor \big), \\
N_2&=\sum_{0<n\leq \beta} \ \sum_{\alpha<m\leq g(n)} 1 = \sum_{0<n\leq \beta} \big( \lfloor g(n)\rfloor-\lfloor \alpha \rfloor \big) , \\
N_3&=\lfloor \alpha \rfloor \lfloor \beta\rfloor .
\end{align*}
In terms of the \emph{sawtooth} function $\psi$, defined by
\[
\psi(x)=x-\lfloor x \rfloor -1/2 ,
\]
one can evaluate 
\[
N_1
=\sum_{0<m\leq \alpha} \big( f(m) -\psi\big(f(m)\big)-1/2-\lfloor \beta \rfloor \big) .
\]
Then we apply the Euler--Maclaurin summation formula 
\[
\sum_{0<m\leq \alpha} f(m)=\int_0^\alpha f(x) \ud x -\psi (\alpha)f(\alpha)+\psi(0)f(0) +\int_0^\alpha f'(x) \psi(x) \ud t
\]
(which we observe for later reference holds whenever $f$ is piecewise $C^1$-smooth) to deduce that
\begin{align*}
N_1
&=\int_0^{\alpha} f(x) \ud x -\psi(\alpha) f(\alpha) +\psi(0) f(0)+\int_0^{\alpha} f'(x) \psi(x) \ud x \\
& \hspace{3cm} - \sum_{0<m\leq \alpha} \psi \big(f(m)\big) - \lfloor \alpha \rfloor (1/2+\lfloor \beta \rfloor) \\
& =\int_0^{\alpha } f(x) \ud x -\psi(\alpha) \beta -M/2+\int_0^{\alpha} f'(x) \psi(x) \ud x \\
& \hspace{3cm} - \sum_{0<m\leq \alpha} \psi \big(f(m)\big) - \lfloor \alpha \rfloor (1/2+\lfloor \beta \rfloor) . 
\end{align*}
Similarly
\begin{align*}
N_2& =\int_0^{\beta } g(y) \ud y -\psi(\beta) \alpha -L/2+\int_0^{\beta} g'(y) \psi(y) \ud y \\
& \hspace{3cm} - \sum_{0<n\leq \beta} \psi \big(g(n)\big) - \lfloor \beta \rfloor (1/2+\lfloor \alpha \rfloor) ,
\end{align*}
and so
\begin{align}
N&=N_1+N_2+N_3\nonumber\\
&=\int_0^{\alpha } f(x) \ud x+\int_0^{\beta } g(y) \ud y - \lfloor \alpha \rfloor \lfloor \beta \rfloor - (L+M)/2 \nonumber\\
& \hspace{1cm} -\psi(\alpha) \beta -\lfloor \alpha \rfloor/2 -\psi(\beta) \alpha -\lfloor \beta \rfloor/2 \nonumber\\
& \hspace{1cm} +\int_0^{\alpha} f'(x) \psi(x) \ud x +\int_0^{\beta} g'(y) \psi(y) \ud y \nonumber \\
& \hspace{1cm}  - \sum_{0<m\leq \alpha} \psi \big(f(m)\big) - \sum_{0<n\leq \beta} \psi \big(g(n)\big) \nonumber\\
&=\area(\Gamma) -(L+M)/2+\int_0^{\alpha} f'(x) \psi(x) \ud x +\int_0^{\beta} g'(y) \psi(y) \ud y \nonumber\\
& \hspace{1cm}- \sum_{0<m\leq \alpha} \psi \big(f(m)\big) - \sum_{0<n\leq \beta} \psi \big(g(n)\big) + \text{remainder} \label{eq:lattice_count}
\end{align}
where 
\begin{equation}\label{eq:const_bound}
\text{remainder} =-(\alpha -\lfloor \alpha \rfloor)(\beta- \lfloor \beta \rfloor) + (\alpha -\lfloor \alpha \rfloor + \beta -\lfloor \beta \rfloor)/2 .
\end{equation}
This remainder lies between $0$ and $1$, since $0 \leq -xy+(x+y)/2 \leq 1$ when $x,y \in [0,1]$. 

We estimate the sum of sawtooth functions in \autoref{eq:lattice_count} by using \autoref{lemma:kratzel} (which is due to van der Corput): since $f''$ is monotonic and nonzero on $[0,\alpha]$, the thoerem implies 
\begin{align}
\Big|\sum_{0<m\leq \alpha }\psi\big(f(m)\big)\Big|& \leq 6 \int_0^{\alpha} |f''(x)|^{1/3} \ud x + 175 \max_{[0,\alpha]} \frac{1}{|f''|^{1/2}}+1 \label{eq:psi_f_bound}
\end{align}
and similarly
\begin{align}\label{eq:psi_g_bound}
\Big|\sum_{0<n\leq \beta }\psi\big(g(n)\big)\Big|& \leq 6\int_0^{\beta } |g''(y)|^{1/3} \ud y + 175 \max_{[0,\beta]} \frac{1}{|g''|^{1/2}}+1 .
\end{align}

To estimate the integrals of $f'\psi$ and $g'\psi$ in \autoref{eq:lattice_count}, we introduce the antiderivative of the sawtooth function, $\Psi(t)=\int_0^t \psi(z) \ud z$, and observe that $-1/8 \leq \Psi(t) \leq 0$ for all $t \in \R$. By integration by parts and the fact that $f''<0$, we have 
\begin{align}\label{eq:int_f_bound}
\big| \int_0^{\alpha} f'(x) \psi(x) \ud x \big|&=\Big| \big[ f'(x) \Psi(x) \big]_{x=0}^{x=\alpha}- \int_0^{\alpha} f''(x)\Psi(x) \ud x\Big| \nonumber\\
&\leq \frac{1}{8} |f'(\alpha)| + \frac{1}{8} \big| \int_0^\alpha f''(x) \ud x \big| \nonumber\\
&= \frac{1}{8}|f'(\alpha)|+\frac{1}{8} \big( f'(0) - f'(\alpha) \big) \nonumber\\
&\leq \frac{1}{4}|f'(\alpha)| 
\end{align}
since $f^\prime(\alpha) \leq f^\prime(0) \leq 0$. The same argument gives 
\begin{equation}\label{eq:int_g_bound}
\big| \int_0^{\beta} g'(y) \psi(y) \ud y \big|\leq \frac{1}{4}|g'(\beta)| .
\end{equation}
Combining \autoref{eq:lattice_count}--\autoref{eq:int_g_bound} completes the proof of Part~(a).

\smallskip
Part (b). Simply apply Part (a) to the curve $r\Gamma(s)$ by replacing $L, M, f(x), g(y), \alpha, \beta$ with $rs^{-1}L, rsM, rsf(sx/r), rs^{-1}g(s^{-1}y/r), rs^{-1}\alpha, rs\beta$ respectively.
\end{proof}

\begin{remark}
\autoref{th:asymptotic} continues to hold if the point $(\alpha,\beta)=(L,0)$ lies at the right endpoint of the curve. One simply removes all mention of $g,\beta_j$ and $\epsilon$ from the hypotheses of the proposition, and removes all such terms from the conclusions, as can be justified by inspecting the proof above. The same remark holds for the advanced counting estimate in the following \autoref{th:asymptotic_general}. 
\end{remark}

\subsection*{Advanced counting estimate}
The hypotheses in the last result are somewhat restrictive. In particular, we would like to handle infinite curvature at the intercepts of the curve $\Gamma$, meaning $f^{\prime \prime}$ must be allowed to blow up at $x=0$. Further, we would like to relax the monotonicity assumption on $f^{\prime \prime}$. The next result achieves these goals. 

Two numbers $\delta$ and $\epsilon$ appear in the next Proposition. Their role in the proof is that on the intervals $0<x \leq \delta$ and $0<y \leq \epsilon$ we bound the sawtooth function trivially with $|\psi| \leq 1/2$. On the remaining intervals we seek cancellations. 
\begin{proposition}[Two-term counting estimate for more general curve]\label{th:asymptotic_general}
Take a point $(\alpha ,\beta) \in \Gamma$ lying in the first quadrant, and assume $f\in PC^2(0,\alpha]$ with $f'<0$ and $f''<0$, and that $f^{\prime \prime}$ is monotonic on $(\alpha_{i-1},\alpha_i]$ for $i=1,\dots,l$. Similarly assume $g\in PC^2(0,\beta]$ with $g'<0$ and $g''<0$, and that $g^{\prime \prime}$ is monotonic on $(\beta_{j-1},\beta_j]$ for $j=1,\dots,\ell$.  

\noindent (a) If $\delta \in (0,\alpha)$ and $\epsilon \in (0,\beta)$ then the number $N$ of positive-integer lattice points inside $\Gamma$ in the first quadrant satisfies:
\begin{align*}
\big|N- & \area (\Gamma)+(L+M)/2\big|\nonumber\\
& \leq 6\Big(\int_0^\alpha |f''(x)|^{1/3} \ud x+\int_0^\beta |g''(y)|^{1/3}\ud y\Big) \nonumber\\
& \quad + 175 \Big(\frac{1}{|f''(\delta)|^{1/2}}+\frac{1}{|g''(\epsilon)|^{1/2}}\Big) + 350 \Big(\sum_{i=1}^l \frac{1}{|f''(\alpha_i)|^{1/2}} + \sum_{j=1}^\ell \frac{1}{|g''(\beta_j)|^{1/2}}\Big) \nonumber \\
& \quad + \frac{1}{4} \big( \sum_{i=1}^l |f'(\alpha_i)| + \sum_{j=1}^\ell |g'(\beta_j)|\big)+\frac{1}{2}\big(\delta+\epsilon\big)+l+\ell+1 . \label{eq:gamma_asymptotic_general}
\end{align*}

\noindent (b) If functions 
\[
\delta : (0,\infty) \to (0,\alpha) , \qquad \epsilon : (0,\infty) \to (0,\beta) ,
\]
are given, then the number $N(r,s)$ of positive-integer lattice points inside $r\Gamma(s)$ in the first quadrant satisfies (for $r,s>0$):
\begin{align}
& \big|N(r,s)- r^2\area (\Gamma)+r(s^{-1}L+sM)/2\big|\nonumber\\
&\leq 6r^{2/3}\Big(\int_0^\alpha |f''(x)|^{1/3} \ud x+\int_0^\beta |g''(y)|^{1/3}\ud y\Big) \nonumber \\
& \quad + 175 r^{1/2}\Big(\frac{s^{-3/2}}{|f'' \big( \delta(r) \big)|^{1/2}}+\frac{s^{3/2}}{|g'' \big( \epsilon(r) \big)|^{1/2}}\Big) + 350r^{1/2} \Big( \sum_{i=1}^l \frac{s^{-3/2}}{|f''(\alpha_i)|^{1/2}} + \sum_{j=1}^\ell \frac{s^{3/2}}{|g''(\beta_j)|^{1/2}}\Big) \nonumber \\
& \quad+\frac{1}{4} \big( \sum_{i=1}^l s^2|f'(\alpha_i)| + \sum_{j=1}^\ell s^{-2}|g'(\beta_j)|\big)+\frac{r}{2}\big(s^{-1}\delta(r)+s\epsilon(r)\big)+l+\ell+1 . \label{eq:gamma_a_asymptotic_general} 
\end{align}
\end{proposition}
The integral of $|f''|^{1/3}$ appearing in the conclusion of \autoref{th:asymptotic_general} is finite, because by H\"{o}lder's inequality and the fact that $f^{\prime \prime}<0$ and $f$ is decreasing, we have
\[
\int_0^{\alpha_1} |f''(x)|^{1/3} \ud x \leq \alpha_1^{2/3} \Big| \int_0^{\alpha_1} f''(x) \ud x \Big|^{\! 1/3} = \alpha_1^{2/3} |f'(0^+)-f'(\alpha_1^-)|^{1/3} < \infty .
\]
The integral of $|g^{\prime \prime}|^{1/3}$ is finite for similar reasons. 
\begin{proof}
Part (a). 
The lattice point counting equation \autoref{eq:lattice_count} holds just as in the proof of \autoref{th:asymptotic}, and so the task is to estimate each of the terms on the right side of that equation. 

Estimate \autoref{eq:psi_f_bound} on the sum of the sawtooth function is no longer valid, because $f''$ is no longer assumed to be monotonic on the whole interval $[0,\alpha]$. To control this sawtooth sum, we first observe 
\[
\big|\sum_{0 < m \leq \delta }\psi\big(f(m)\big)\big| \leq \frac{1}{2} \delta
\]
since $|\psi| \leq 1/2$ everywhere. Next, we have $\delta \in (\alpha_{j-1},\alpha_j]$ for some $j \in \{ 1,\ldots,l \}$, and  
\[
\big|\sum_{\delta < m \leq \alpha_j}\psi\big(f(m)\big)\big| \leq 6 \int_\delta^{\alpha_j} |f''(x)|^{1/3} \ud x + 175 \max \Big\{ \frac{1}{|f''(\delta)|^{1/2}} , \frac{1}{|f''(\alpha_j)|^{1/2}} \Big\} + 1 
\]
by \autoref{lemma:kratzel} applied on the interval $[\delta,\alpha_j]$. Applying that theorem again on each interval $[\alpha_{i-1},\alpha_i]$ with $i=j+1,\ldots,l$ gives that
\[
\big|\sum_{\alpha_{i-1} < m \leq \alpha_i}\psi\big(f(m)\big)\big| \leq 6 \int_{\alpha_{i-1}}^{\alpha_i} |f''(x)|^{1/3} \ud x + 175 \max \Big\{ \frac{1}{|f''(\alpha_{i-1})|^{1/2}} , \frac{1}{|f''(\alpha_i)|^{1/2}} \Big\} + 1 .
\]
By summing the last three displayed inequalities, we deduce a sawtooth bound
\begin{align}
& \Big|\sum_{0<m\leq \alpha }\psi\big(f(m)\big)\Big| \notag \\
& \leq \frac{1}{2} \delta+ 6 \int_\delta^\alpha |f''(x)|^{1/3} \ud x + \frac{175}{|f''(\delta)|^{1/2}} +  \sum_{i=j}^{l-1} \frac{350}{|f''(\alpha_i)|^{1/2}} + \frac{175}{|f''(\alpha)|^{1/2}} + l-j+1 \nonumber \\
& \leq \frac{1}{2} \delta+ 6 \int_0^\alpha |f''(x)|^{1/3} \ud x + \frac{175}{|f''(\delta)|^{1/2}} +  \sum_{i=1}^l \frac{350}{|f''(\alpha_i)|^{1/2}} + l . \label{eq:psi_f_bound_gd}
\end{align}
%

Next, we adapt estimate \autoref{eq:int_f_bound} on the integral of $f^\prime \psi$ by simply applying the same argument on each interval $[\alpha_{i-1},\alpha_i]$, hence finding 
\begin{align}\label{eq:int_f_bd_gd}
\Big| \int_0^{\alpha} f'(x) \psi(x) \ud x \Big|
&\leq \sum_{i=1}^l \Big[ \frac{1}{8}|f'(\alpha_i)|+\frac{1}{8} \big( f'(\alpha_{i-1})-f'(\alpha_i) \big) \Big] \nonumber\\
&\leq \frac{1}{4}\sum_{i=1}^l|f'(\alpha_i)|.
\end{align}
%
%
By combining \autoref{eq:lattice_count}, \autoref{eq:const_bound} with \autoref{eq:psi_f_bound_gd}, \autoref{eq:int_f_bd_gd} and the analogous estimates on $g$, we complete the proof of Part~(a).

\smallskip
Part (b).
Apply Part (a) to the curve $r\Gamma(s)$ by replacing $L, M, f(x), g(y), \alpha, \beta,\delta,\epsilon$ with $rs^{-1}L, rsM, rsf(sx/r), rs^{-1}g(s^{-1}y/r), rs^{-1}\alpha, rs\beta,rs^{-1}\delta(r),rs\epsilon(r)$ respectively.
\end{proof}

\section{\bf A unified approach}
\label{sec:structure}

The next proposition provides a unified framework for proving our theorems later in the paper. It adapts the scheme of proof employed by Antunes and Freitas \cite{AF13}.

\begin{proposition}\label{prop:unified}
Let $A \in \R, L>0$, and $0<\theta<1$. Consider a real valued function $H(r,s)$ (for $r,s>0$) such that for each closed interval $[s_1,s_2] \subset (0,\infty)$ one has 
\begin{equation}\label{eq:two-term-ineq}
H(r,s) = Ar^2 - Lr(s+s^{-1})/2+O(r^\theta),
\end{equation}
with $s \in [s_1,s_2]$ allowed to vary as $r \to \infty$. 
Assume the function $s \mapsto H(r,s)$ attains its maximum value, for each $r>0$, and write $S(r) = \argmax_{s>0} H(r,s)$ for the set of maximizing points. Suppose 
\begin{equation}\label{eq:two-term-upper-bound}
S(r) \subset [s_1,s_2]  \qquad \text{for all large $r>0$,}
\end{equation}
for some constants $s_1,s_2>0$. 

Then the maximizing set $S(r)$ converges to the point $\{ 1 \}$ as $r \to \infty$, with 
\[
S(r) \subset \big[1-O(r^{-(1-\theta)/2}),1+O(r^{-(1-\theta)/2})\big], 
\]
and the maximum value of $H$ has asymptotic formula
\[
\max_{s > 0} H(r,s) = Ar^2-Lr + O(r^{\theta}) .
\]
\end{proposition}
The error term $O(r^\theta)$ in \eqref{eq:two-term-ineq} has implied constant depending on the interval $[s_1,s_2]$.
\begin{proof}
Since $S(r) \subset [s_1,s_2]$ by hypothesis \autoref{eq:two-term-upper-bound}, the asymptotic estimate \autoref{eq:two-term-ineq} implies  
\begin{align*}
H(r,s) & = Ar^2-Lr(s+s^{-1})/2 + O(r^{\theta}) , \\
H(r,1) & = Ar^2-Lr + O(r^{\theta}),
\end{align*}
for $s \in S(r)$ and $r \to \infty$. Since $s$ is a maximizing value, we have $H(r,1)\leq H(r,s)$ and so
\begin{equation} \label{eq:s-sinverse}
s+s^{-1} \leq 2 +O(r^{-(1-\theta)}) .
\end{equation}
Hence $s=1+O(r^{-(1-\theta)/2})$ by \autoref{le:squarecompletion}, which proves the first claim in the theorem. For the second claim, when $s \in S(r)$ we have
$
H(r,s) = Ar^2-Lr + O(r^{\theta}) 
$
as $r \to \infty$, by \autoref{eq:two-term-ineq} and using also that $1 \leq (s+s^{-1})/2 \leq 1+O(r^{-(1-\theta)})$ by \autoref{eq:s-sinverse}.
\end{proof}

\section{\bf Proof of \autoref{th:S_limit} and \autoref{th:S_limit_general}}
\label{sec:mainproof}

\subsection*{Proof of \autoref{th:S_limit}}
The theorem follows directly from \autoref{prop:unified} with $H(r,s)$ being the lattice counting function $N(r,s)$. The hypotheses of the proposition are verified as follows. 

Suppose $0<s_1<s_2<\infty$. By \autoref{th:asymptotic}(b) with $L=M$ one has 
\begin{equation}\label{eq:two-term-ineqN}
N(r,s) = \area(\Gamma)r^2 - Lr(s+s^{-1})/2+O(r^{2/3}),
\end{equation}
with $s \in [s_1,s_2]$ as $r \to \infty$. Thus hypothesis \eqref{eq:two-term-ineq} holds for $N(r,s)$ with the choices $A=\area(\Gamma), \theta=2/3$, and $L$ equalling the intercept value of $\Gamma$. 

The boundedness hypothesis \eqref{eq:two-term-upper-bound} holds by \autoref{thm:s_bounded}. 

\subsection*{Proof of \autoref{th:S_limit_general}}
\label{sec:generalproof}
Again let $H(r,s)$ be the lattice counting function $N(r,s)$, take $A=\area(\Gamma)$, let $L$ be the intercept value of $\Gamma$, and note the boundedness hypothesis \eqref{eq:two-term-upper-bound} holds by \autoref{thm:s_bounded}. To finish verifying the hypotheses of \autoref{prop:unified}, we suppose $0<s_1<s_2<\infty$ and show that \autoref{eq:two-term-ineq} holds. 

Take $\theta=1-2e$, where the number $e=\min \{ \tfrac{1}{6},a_1,a_2,b_1,b_2 \}$ was defined in \autoref{th:S_limit_general}. Hypothesis \eqref{eq:two-term-ineq} is the assertion that
\begin{equation}\label{eq:two-term-ineqNgeneral}
N(r,s) = \area(\Gamma)r^2 - Lr(s+s^{-1})/2+O(r^{1-2e}),
\end{equation}
with $s \in [s_1,s_2]$ as $r \to \infty$. To verify this asymptotic, we will estimate the remainder terms in \autoref{th:asymptotic_general}(b) as follows. In that proposition take $L=M$, and note $\delta(r)<\alpha$ and $\epsilon(r)<\beta$ for all large $r$ by assumptions \autoref{eq:f_sup} and \autoref{eq:g_sup}. We will show the right side of estimate \autoref{eq:gamma_a_asymptotic_general} in \autoref{th:asymptotic_general}(b) is bounded by
\begin{align*}
& O(r^{2/3}) + s^{-3/2} O(r^{1-2a_2}) + s^{3/2} O(r^{1-2b_2}) + (s^{-3/2}+s^{3/2}) O(r^{1/2}) \\
& \quad \qquad + (s^2+s^{-2}) O(1) + s^{-1} O(r^{1-2a_1}) + s O(r^{1-2b_1}) + O(1)
\end{align*}
for large enough $r$, where the implied constants in the $O(\cdot)$-terms depend only on the curve $\Gamma$ and are independent of $s$. Since each one of these $O(\cdot)$-terms is bounded by $O(r^{1-2e})$, and $s$ and $s^{-1}$ are bounded when $s \in [s_1,s_2]$, hypothesis \eqref{eq:two-term-ineq} will hold as desired. 

Examining now the right side of \autoref{eq:gamma_a_asymptotic_general}, we see the first two terms are obviously $O(r^{2/3})$. For the next term, observe by assumption in \autoref{eq:f_sup} that 
\[
\frac{r^{1/2} s^{-3/2}}{|f''(\delta(r))|^{1/2}}=s^{-3/2}O(r^{1-2a_2}) ,
\]
and similarly for the analogous term involving $g''$. Since $f''(\alpha_i)$ and $g''(\beta_j)$ are constant, the corresponding terms in \autoref{eq:gamma_a_asymptotic_general} can be estimated by $(s^{-3/2}+s^{3/2}) O(r^{1/2})$. Similarly, the terms in \eqref{eq:gamma_a_asymptotic_general} involving $f'(\alpha_i)$ and $g'(\beta_j)$ can be estimated by $(s^2+s^{-2}) O(1)$. Next, $s^{-1}r\delta(r)=s^{-1}O(r^{1-2a_1})$ by the assumption in \autoref{eq:f_sup}, and similarly for $\epsilon(r)$. And, of course, $l+\ell+1$ is constant, which completes the verification of hypothesis \autoref{eq:two-term-ineq}.

\section{\bf Proof of \autoref{thm:s_bounded_neumann} and \autoref{th:S_limit_neumann}}
\label{sec:closedquadrant}

First we need a two-term bound on the counting function in the closed first quadrant, as provided by the next proposition. The result is an analogue of \autoref{prop:counting_ineq}, although the constant $\mathcal{C}$ is slightly different than in that result. 

Assume $f$ is concave and strictly decreasing on $[0,L]$, with $f(0)=M, f(L)=0$. The intercepts $L$ and $M$ need not be equal. 

\begin{proposition}[Two-term lower bound on counting function]\label{prop:counting_ineq_neumann}
Let $\mathcal{C}=M-f(L/4)$.

\noindent (a) The number $\mathcal{N}$ of nonnegative-integer lattice points lying inside $\Gamma$ in the closed first quadrant satisfies:
\begin{equation}
\mathcal{N} \geq \area(\Gamma) + \frac{1}{2} \mathcal{C} .
\end{equation}

\noindent (b) The number of nonnegative-integer lattice points lying inside $r\Gamma(s)$ in the closed first quadrant satisfies (for $r,s>0$):
\[
\mathcal{N}(r,s)\geq r^2\area(\Gamma) + \frac{1}{2}\mathcal{C}rs .
\]
\end{proposition}
\begin{proof}
Part (a). 
Clearly $\mathcal{N}$ equals the total area of the squares of sidelength $1$ having lower left vertices at nonnegative-integer lattice points inside the curve $\Gamma$. The union of these squares contains $\Gamma$, since the curve is decreasing.

We separate the proof into cases according to the value of $L$.

Case (i): Suppose $L\leq 2$, so that $L/4 \leq 1/2$. Consider a rectangle whose lower left vertex sits on the curve at $x=L/4$, and has vertices
\[
\big( L/4,f(L/4) \big), \quad \big( 1, f(L/4) \big), \quad \big( 1, M \big), \quad \big( L/4, M \big) .
\] 
By construction, this rectangle lies inside the union of squares of sidelength $1$, and it lies above $\Gamma$ because the curve is decreasing. Hence
\begin{align*}
\mathcal{N} 
& \geq \area(\Gamma) + \area(\text{rectangle}) \\
& = \area(\Gamma) +(1-L/4)\big(M - f(L/4)\big) \\
& \geq \area(\Gamma)+\frac{1}{2}\big(M-f(L/4)\big) 
\end{align*}
as desired.

Case (ii): Suppose $L\geq 2$. Consider the right triangles of width $1$ formed by tangent lines from the right on $\Gamma$, that is, the triangles with vertices $\big( i,f(i) \big),\big( i+1,f(i) \big),\big( i+1,f(i)+f'(i^+) \big)$, where $i=0,1,\ldots,\lfloor L \rfloor-1$. These triangles all lie above the horizontal axis, since by concavity $f(i)+f'(i^+) \geq f(i+1) \geq 0$; the last inequality explains why the biggest $i$-value we consider is $\lfloor L \rfloor - 1$. 

\begin{figure}
\includegraphics[scale=0.4]{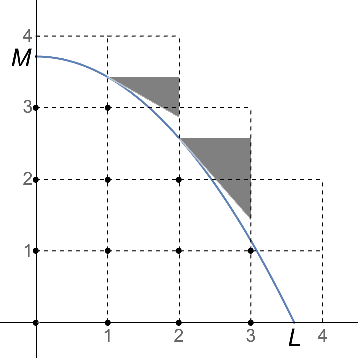}
\caption{\label{fig:neumann_counting}Nonnegative integer lattice count satisfies $\mathcal{N} \geq \area(\Gamma) + \area(\text{triangles})$, in proof of \autoref{prop:counting_ineq_neumann}(a) when $L \geq 2$.}
\end{figure}

Thus these triangles lie inside the union of squares of sidelength $1$, and lie above $\Gamma$ by concavity. Hence
\[
\mathcal{N} \geq \area(\Gamma) + \area(\text{triangles}) .
\]
To complete the proof of Case (ii), we estimate
\begin{align*}
\area(\text{triangles}) 
& \geq \frac{1}{2} \sum_{i=1}^{\lfloor L \rfloor - 1} |f'(i^+)| \\
& \geq \frac{1}{2} \sum_{i=1}^{\lfloor L \rfloor - 1} \big( f(i-1)-f(i) \big) \qquad \text{by concavity} \\
& = \frac{1}{2} \big( f(0) - f(\lfloor L \rfloor - 1) \big) \\
& \geq \frac{1}{2}\big( M-f(L/4) \big) ,
\end{align*}
because $\lfloor L \rfloor-1 \geq L/2 \geq L/4$ when $L \geq 2$. 

\smallskip
Part (b).
Replace $\Gamma$ in Part (a) with the curve $r\Gamma(s)$, meaning we replace $L, M, f(x)$ with $rs^{-1}L, rsM, rsf(sx/r)$ respectively.
\end{proof}

\subsection*{Proof of \autoref{thm:s_bounded_neumann}} Since $N(r,s)\leq r^2\area(\Gamma)$, taking $s=1$ and $L=M$ in \autoref{le:relation} gives that
\[
\mathcal{N}(r,1)\leq r^2 \area(\Gamma) + 2rL + 1.
\]
Now suppose $s\in \mathcal{S}(r)$ is a minimizing value, so that $\mathcal{N}(r,s)\leq \mathcal{N}(r,1)$. Since 
\[
\mathcal{N}(r,s) \geq r^2 \area(\Gamma) +\frac{1}{2}\mathcal{C}rs
\]
by \autoref{prop:counting_ineq_neumann}(b), we conclude from above that
\[
\frac{1}{2•}\mathcal{C}rs \leq 2rL+1\leq \frac{5}{2•}rL ,
\]
where the last inequality holds for $r \geq 2/L$. Hence $s \leq 5L/\mathcal{C}$, and so the set $\mathcal{S}(r)$ is bounded above. Interchanging the horizontal and vertical axes and recalling $L=M$ (\emph{i.e.}, the intercepts are equal in this theorem), one finds similarly that $s^{-1}\leq 5L/\widetilde{\mathcal{C}}$. Hence $\mathcal{S}(r)$ is bounded below away from $0$, which completes the proof.

\subsection*{Proof of \autoref{th:S_limit_neumann}}
The theorem will follow from \autoref{prop:unified} with the choice $H(r,s)=-\mathcal{N}(r,s)$, since maximizing $s \mapsto H(r,s)$ corresponds to minimizing $s \mapsto \mathcal{N}(r,s)$. The boundedness hypothesis \eqref{eq:two-term-upper-bound} of the proposition holds by \autoref{thm:s_bounded_neumann}. The other hypothesis \eqref{eq:two-term-ineq} is verified as follows. 

Taking $L=M$ in the relation between $\mathcal{N}(r,s)$ and $N(r,s)$ in \autoref{le:relation}, and calling on the asymptotic for $N(r,s)$ in either \autoref{eq:two-term-ineqN} (under the assumptions of \autoref{th:S_limit}) or \autoref{eq:two-term-ineqNgeneral} (under the assumptions of \autoref{th:S_limit_general}), we deduce
\[
\mathcal{N}(r,s) = \area(\Gamma)r^2+Lr(s^{-1}+s)/2+ O(r^\theta)
\]
with $s \in [s_1,s_2]$ allowed to vary as $r \to \infty$, where
\[
\theta = 
\begin{cases}
2/3 & \text{under the assumptions of \autoref{th:S_limit},} \\
1-2e & \text{under the assumptions of \autoref{th:S_limit_general}.}
\end{cases}
\]
That is, we have verified hypothesis \eqref{eq:two-term-ineq} with  $H=-\mathcal{N},A=-\area(\Gamma)$, and $L$ the intercept value of $\Gamma$.

\section{\bf Open problem for $1$-ellipses --- lattice points in right triangles}
\label{sec:oneellipse}
 
Lattice point maximization for right triangles appears to be an open problem. Consider the $p$-circle with $p=1$, which is a diamond with vertices at $(\pm 1,0)$ and $(0,\pm 1)$. It intersects the first quadrant in the line segment $\Gamma$ joining the points $(0,1)$ and $(1,0)$. Here $L=M=1$. Stretching the $1$-circle in the $x$- and $y$-directions gives a $1$-ellipse 
\[
|sx|+|s^{-1}y|=1 ,
\]
which together with the coordinate axes forms a right triangle of area $1/2$ in the first quadrant, with one vertex at the origin and hypotenuse $\Gamma(s)$ joining the vertices at $(s^{-1},0)$ and $(0,s)$. As previously, we write $S(r)$ for the set of $s$-values that maximize the number of positive-integer (first quadrant) lattice points below or on $r\Gamma(s)$, when $r>0$. 

First of all, the 45--45--90 degree triangle ($s=1$) does not always enclose the most lattice points: \autoref{fig:counterexample_p=1} shows an example.

\begin{figure}
\includegraphics[scale=0.35]{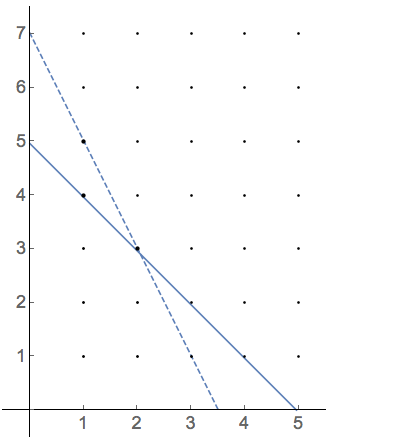}
\caption{\label{fig:counterexample_p=1}The $1$-ellipse $sx+s^{-1}y=r$ with $r=4.96$, for  $s=1$ (solid) and $s=\sqrt{2}$ (dashed). The dashed line encloses three more lattice points (shown in bold) than the solid line.}
\end{figure}

The open problem is to understand the limiting behavior of the maximizing $s$-values. Does $S(r)$ converge to $\{ 1 \}$ as $r \to \infty$? We proved the answer is ``Yes'' for $p$-ellipses when $1<p<\infty$ (\autoref{ex:p-ellipse}), but for $p=1$ we suggest the answer is ``No''. Numerical evidence in \autoref{fig:optimal_s_p_1} suggests that the set $S(r)$ does not converge to $\{ 1 \}$ as $r \to \infty$. Indeed, the plotted heights appear to cluster at a large number of values, possibly dense in some interval around $s=1$. These cluster values presumably have some number theoretic significance. 

\begin{figure}
\includegraphics[scale=0.45]{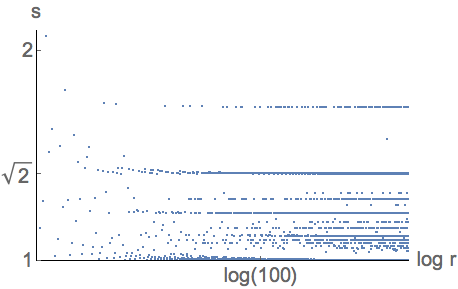}
\caption{\label{fig:optimal_s_p_1}Optimal $s$-values for maximizing the number of lattice points in the $1$-ellipse (triangle). The graph plots $\sup S(r)$ versus $\log r$. The plotted $r$-values are multiples of $\sqrt{3}/10$, and the horizontal axis is at height $s=1$.}
\end{figure}

In the remainder of the section we remark that maximizing $s$-values are $\leq 3$ in the limit as $r \to \infty$, and we describe the numerical scheme that generates \autoref{fig:optimal_s_p_1}. Lastly, we explain why $s= \sqrt{2}$ is a good candidate for a cluster value as $r \to \infty$.

\subsection*{The bound on maximizing $s$-values for right triangles ($p=1$)} Given $\e>0$, we claim 
\[
S(r) \subset \big[\frac{1}{3+\e},3+\e\big] \qquad \text{for all large $r$.}
\]
This bound is slightly better than the one in \autoref{thm:s_bounded} (which had $4$ instead of $3$), and can be proved in the same way with the help of a special formula for $N(r,1)$:
\begin{align}
N(r,1) 
&= \# \ \text{first-quadrant lattice points under the line $y=r-x$}\nonumber\\
&= \frac{1}{2} \lfloor r \rfloor \lfloor r -1\rfloor \nonumber\\
&\geq  \frac{1}{2}(r -1)(r -2) 
= \frac{1}{2} r^2 -\frac{3}{2}r + 1.\label{eq:p_1_lower_bound}
\end{align}

\subsection*{How can one efficiently maximize the lattice counting function for the $1$-ellipse?} A brute force method of counting how many lattice points lie under the line $r\Gamma(s)$, and then varying $s$ to maximize that number of lattice points, is simply unworkable in practice. The counting function $N(r,s)$ jumps up and down in value as $s$ varies, sometimes jumping quite rapidly, and a brute force method of sampling at a finite collection of $s$-values can never be expected to capture all such jump points or their precise locations. 

Instead, for a given $r$ we should pre-identify the possible jump values of $s$, and use that information to count the lattice points. We start with the simple observation that a lattice point $(j,k)$ lies under the line $r\Gamma(s)$ if and only if
\[
sj+s^{-1}k \leq r ,
\] 
which is equivalent to
\begin{equation}\label{eq:s_quadratic}
js^2 -r s + k \leq 0 .
\end{equation}
For this quadratic inequality to have a solution, the discriminant must be nonnegative, $r^2-4jk\geq 0$, and thus we need only consider lattice points beneath the hyperbola $r^2=4xy$. For each such lattice point, equality holds in \autoref{eq:s_quadratic} for two positive $s$-values, namely 
\[
s_{min}(j,k;r) = \frac{r-\sqrt{r^2-4jk}}{2j} , \quad
s_{max}(j,k;r) = \frac{r+\sqrt{r^2-4jk}}{2j}.
\]
The geometrical meaning of these values can be understood, as follows: as $s$ increases from $0$ to $\infty$, one endpoint of the line segment $r\Gamma(s)$ slides up on the $y$-axis while the other endpoint moves left on the $x$-axis. The line segment passes through the point $(j,k)$ twice: first when $s=s_{min}(j,k;r)$ and again when $s=s_{max}(j,k;r)$. The point $(j,k)$ lies below the line when $s$ belongs to the closed interval between these two values.  

Thus the counting function is
abad
\begin{align*}
N(r,s)
&=\# \big\{ (j,k):s_{min}(j,k;r) \leq s \leq s_{max}(j,k;r) \big\} \\
& = \sum_{j,k>0}\mathbbm{1}_{s_{min}(j,k;r)\leq s}-\sum_{j,k>0}\mathbbm{1}_{s_{max}(j,k;r)<s} 
\end{align*}
where we sum only over positive-integer lattice points with $4jk \leq r^2$.

The last formula says that the counting function $N(r,s)$ equals the number of values $s_{min}(j,k;r)$ that are less than or equal to $s$ minus the number of values $s_{max}(j,k;r)$ that are less than $s$. To facilitate the evaluation in practice, one should sort the list of values of $s_{min}(j,k;r)$ into increasing order, and similarly sort the list of values of $s_{max}(j,k;r)$. The numbers in these two lists are the only numbers where $N(r,s)$ can change value, as $s$ increases. In particular, when $s$ increases to $s_{min}(j,k;r)$, the point $(j,k)$ is picked up by the line segment for the first time and so $N(r,s)$ increases by $1$. When $s$ increases strictly beyond $s_{max}(j,k;r)$, the point $(j,k)$ is dropped by the line segment and so $N(r,s)$ decreases by $1$. Note the counting function might increase or decrease by more than $1$ at some $s$-values, if the sorted lists of $s_{min}$ and $s_{max}$ values have repeated entries (arising from lattice points that are picked up by, or else dropped by, the line segment at the same $s$-value).  

After sorting the $s_{min}$ and $s_{max}$ lists, we evaluate the maximum of $N(r,s)$ by scanning through the two lists, increasing a counter by $1$ at each number in the sorted $s_{min}$ list, and decreasing the counter just after each number in the sorted $s_{max}$ list. The largest value achieved by the counter is the maximum of $N(r,s)$, and $S(r)$ consists of the closed interval or intervals of $s$-values on which this maximum count is attained. 

By this method, we can maximize the lattice counting function for the $1$-ellipse in a computationally efficient manner, for any given $r>0$. The code is available in \cite[Appendix~B]{thesis}. 

When presenting the results of this method graphically, in \autoref{fig:optimal_s_p_1}, we plot only the largest $s$ value in $S(r)$, because the family of $1$-ellipses is invariant under the map $s \mapsto 1/s$ and so the smallest value in $S(r)$ will be just the reciprocal of the largest value.

\subsection*{Why is the $1$-ellipse not covered by our theorems?} For the $p$-ellipse with $p=1$, \autoref{th:S_limit_general} does not apply because $f$ is linear and so $f^{\prime \prime} \equiv 0$. Specifically, in the proof we see inequalities \autoref{eq:psi_f_bound} and \autoref{eq:psi_g_bound} are no longer useful, since their right sides are infinite. The situation cannot easily be rescued, because the left side of \autoref{eq:gamma_a_asymptotic} need not even be $o(r)$. For example, when $s=1$ and $r$ is an integer, by evaluating the number $N(r,1)$ of lattice points under the curve $y=r-x$ we find 
\[
N(r,1) - r^2 \area(\Gamma) + r(L+M)/2
= \frac{1}{2} r(r-1) - \frac{1}{2} r^2 + r = \frac{1}{2} r ,
\]
which is of order $r$ and hence has the same order as the ``boundary term''  $r(L+M)/2$ on the left side. Thus the method breaks down completely for $p=1$. We seek instead to illuminate the situation through numerical investigations.
 
\subsection*{A cluster value at $s = \sqrt{2}${\,}?} 
Inspired by the numerical calculations in \autoref{fig:optimal_s_p_1}, we will show that $s=\sqrt{2}$ gives a substantially higher count of lattice points than $s=1$, for a certain sequence of $r$-values tending to infinity. This observation suggests (but does not prove) that $\sqrt{2}$ or some number close to it should belong to $S(r)$ for those $r$-values. To be clear: we have not found a proof of this claim. Doing so would provide a counterexample to the idea that the set $S(r)$ converges to $\{ 1 \}$ as $r \to \infty$. 

To compare the counting functions for $s=1$ and $s=\sqrt{2}$, we first notice that for $s=1$ the counting function for the $1$-circle is given by
\[
N(r,1) = \lfloor r \rfloor \lfloor r-1 \rfloor / 2  , \qquad r>0 .
\]
At $s=\sqrt{2}$ the slope of the $1$-ellipse is $-2$, and for the special choice $r=\sqrt{2}(m+1/2)$ with $m \geq 1$ the counting function can be evaluated explicitly as
\[
N(r,\sqrt{2}) = m^2 .
\]
We further choose $m$ such that $r \in (n-1/4,n)$ for some integer $n$, noting that an increasing sequence of such $m$-values can be found due to the density in the unit interval of multiples of $\sqrt{2}$ modulo $1$. Then, writing $r=n-\epsilon$ where $\epsilon<1/4$, we have
\begin{align*}
N(r,\sqrt{2})-N(r,1) 
& = m^2 - (n-1)(n-2)/2 \\
& = \frac{1}{2} (r^2-\sqrt{2}r+1/2) - \frac{1}{2} (r+\epsilon-1)(r+\epsilon-2) \\
& \geq \frac{1}{2}r - (\text{constant}) .
\end{align*}
Hence $\limsup_{r \to \infty} \big( N(r,\sqrt{2})-N(r,1) \big)/r \geq 1/2$, and so $s=\sqrt{2}$ can give (for certain choices of $r$) a substantially higher count of lattice points than $s=1$, as we wanted to show. 

The work above implies that $1 \notin S(r)$ for a sequence of $r$-values tending to infinity. More generally, Marshall and Steinerberger  showed that if $x>0$ is rational then $\sqrt{x}\notin S(r)$ for a sequence of $r$-values tending to infinity (see \cite[Theorem~1]{marshall_steinerberger}), while if $x>0$ is irrational then $\sqrt{x}\notin S(r)$ for all sufficiently large $r$ (see \cite[Lemma~2]{marshall_steinerberger} and its associated discussion).

\subsection*{Conjecture for $p=1$} 
To finish the chapter, we state some of our numerical observations as a conjecture. Let 
\begin{align*}
S &= \{ (r,s): r>0, s \in S(r)\} \subset (0,\infty) \times (0,\infty),\\
\overline{S} &= \text{closure of $S$ in $[0,\infty] \times [0,\infty]$}, \\
S(\infty) & = \{ s\in [0, \infty]: (\infty,s) \in \overline{S}\}.
\end{align*}
Earlier in the chapter we proved that $S(\infty) \subset [1/3,3]$. 

The clustering behavior of $S(r)$ observed in \autoref{fig:optimal_s_p_1} suggests the following conjecture.  
\begin{conjecture}[$p=1$] \label{conj:p_1}
The limiting set $S(\infty)$ is countably infinite, and is contained in 
\[
[1/3,3] \cap \{\sqrt{x}:x\in \mathbb{Q}, x>0 \}.
\]
\end{conjecture}
Marshall and Steinerberger \cite[Theorem~2]{marshall_steinerberger} recently proved that $S(\infty)$ contains (countably) infinitely many square roots of rational numbers and is contained in $[1/\sqrt{5},\sqrt{5}]$. For example, they showed the set $S(\infty)$ contains $1$ and $\sqrt{3/2}$. Yet a precise characterization of the set remains elusive. One would like a characterization in terms of some number theoretic condition.

\section{\bf Connection between counting function maximization and eigenvalue minimization} 
\label{sec:relation}

Maximizing a counting function is morally equivalent to minimizing the size of the things being counted. Let us apply this general principle to the case of the circle 
\[
\Gamma:x^2+y^2=1 \quad \text{in the first quadrant,}
\]
and its associated ellipses $r\Gamma(s)$. In this section, $L=M=1$ and $\area (\Gamma) = \pi/4$.

\subsection*{Minimizing eigenvalues of the Dirichlet Laplacian on rectangles}
Write 
\begin{equation} \label{eq:Dirichleteigen}
\{ \lambda_n(s) : n=1,2,3,\ldots \} = \{ (js)^2+(ks^{-1})^2 : j,k =1,2,3,\ldots \} 
\end{equation}
so that $\lambda_n(s)$ is the $n$th eigenvalue of the Dirichlet Laplacian on a rectangle of side lengths $s^{-1}\pi$ and $s\pi$. (The eigenfunctions have the form $\sin(jsx) \sin(ks^{-1}y)$.) Then the lattice point counting function is the eigenvalue counting function, because
\begin{align*}
N(r,s) 
& = \# \{ (j,k) :  (js)^2+(ks^{-1})^2 \leq r^2 \} \\
& = \# \{ n : \lambda_n(s) \leq r^2 \} .
\end{align*}

Define 
\[
S_*(n) = \argmin_{s>0} \lambda_n(s) ,
\]
so that $S_*(n)$ is the set of $s$-values that minimize the $n$th eigenvalue. 

The next result says that the rectangle minimizing the $n$th eigenvalue will converge to a square as $n \to \infty$. 
\begin{corollary}[Optimal Dirichlet rectangle is asymptotically balanced, due to Antunes and Freitas \protect{\cite[Theorem~2.1]{AF13}}; Gittins and Larson \protect{\cite{GL17}}] \label{co:relationDirichlet} \ 

\noindent The optimal stretch factor for minimizing $\lambda_n(s)$ approaches $1$ as $n \to \infty$, with 
\[
S_*(n) \subset [1-O(n^{-1/12}),1+O(n^{-1/12})] ,
\]
and the minimal Dirichlet eigenvalue satisfies the asymptotic formula 
\[
\min_{s > 0} \lambda_n(s) = \frac{4}{\pi} n + \Big( \frac{4}{\pi} \Big)^{\! 3/2} n^{1/2} + O(n^{1/3}) .
\]
\end{corollary}
The proof is a modification of our \autoref{th:S_limit}. Full details are provided in the ArXiv version of this paper \cite[Corollary 10]{Lau_Liu}. In the proof one relies on \autoref{prop:counting_ineq} to bound the stretch factor $s$ of the optimal rectangle. \autoref{prop:counting_ineq} is simpler in both statement and proof than the corresponding Theorem~3.1 of Antunes and Freitas \cite{AF13}, which contains an additional lower order term with an unhelpful sign. 

\begin{remark}
One would like to prove using only the definition of the counting function that 
\[
\text{$S_*(n) \to 1$ \quad if and only if \quad $S(r) \to 1$,}
\]
or in other words that the rectangle minimizing the $n$th eigenvalue will converge to a square if and only if the ellipse maximizing the number of lattice points converges to a circle. Then \autoref{co:relationDirichlet} would follow qualitatively from \autoref{th:S_limit}, without needing any additional proof. Our attempts to find such an abstract equivalence have failed due to possible multiplicities in the eigenvalues. Perhaps an insightful reader will see how to succeed where we have failed. 
\end{remark}

\subsection*{Maximizing eigenvalues of the Neumann Laplacian on rectangles}
If one considers lattice points in the closed first quadrant, that is, allowing also the lattice points on the axes, then one obtains the Neumann eigenvalues of the rectangle having side lengths $s^{-1}\pi$ and $s\pi$:
\[
\{ \mu_n(s) : n=1,2,3,\ldots \} = \{ (js)^2+(ks^{-1})^2 : j,k =0,1,2,\ldots \} .
\]
Notice the first eigenvalue is always zero: $\mu_1(s)=0$ for all $s$. The lattice point counting function is once again an eigenvalue counting function, because
\[
\mathcal{N}(r,s) 
= \# \{ (j,k) :  (js)^2+(ks^{-1})^2 \leq r^2 \} = \# \{ n : \mu_n(s) \leq r^2 \} .
\]
The appropriate problem is to maximize the $n$th eigenvalue (rather than minimizing as in the Dirichlet case), and so we let 
\[
\mathcal{S}_*(n) = \argmax_{s>0} \mu_n(s) .
\]

The corollary below says that the rectangle maximizing the $n$th Neumann eigenvalue will converge to a square as $n \to \infty$. 
\begin{corollary}[Optimal Neumann rectangle is asymptotically balanced, due to van den Berg, Bucur and Gittins \protect{\cite{BBG16a}}; Gittins and Larson \protect{\cite{GL17}}] \label{co:relationNeumann} \ 

\noindent The optimal stretch factor for maximizing $\mu_n(s)$ approaches $1$ as $n \to \infty$, with
\[
\mathcal{S}_*(n) \subset [1-O(n^{-1/12}),1+O(n^{-1/12})] ,
\]
and the maximal Neumann eigenvalue satisfies the asymptotic formula 
\[
\max_{s > 0} \mu_n(s) = \frac{4}{\pi} n - \Big( \frac{4}{\pi} \Big)^{\! 3/2} n^{1/2} + O(n^{1/3}) .
\]
\end{corollary}
One adapts the arguments used for \autoref{th:S_limit_neumann}. A complete proof is in \cite[Corollary 11]{Lau_Liu}. Note that our lower bound on the counting function in \autoref{prop:counting_ineq_neumann}, which one uses to control the stretch factor $s$ of the optimal rectangle, is simpler in both statement and proof than the corresponding Lemma~2.2 by van den Berg \emph{et al.}\ \cite{BBG16a}. Further, our \autoref{prop:counting_ineq_neumann} holds for all $r>0$, whereas \cite[Lemma~2.2]{BBG16a} holds only for $r \geq 2s$. Consequently we need not establish an \emph{a priori} bound on $s$ as was done in \cite[Lemma~2.3]{BBG16a}. 

Those authors did obtain a slightly better rate of convergence than we do, by calling on sophisticated lattice counting estimates of Huxley; see the comments after \autoref{th:asymptotic}.

\appendix 
\section{\bf The van der Corput sum} \label{app-exponential}

The next thoerem is due to van der Corput \cite[Satz~5]{vdC23}, and is central to the proofs of \autoref{th:asymptotic} and \autoref{th:asymptotic_general}. We formulate the theorem as in Kr\"atzel \cite[Korollar zu Satz 1.5, p.~24]{kratzel00}. The constants in the theorem are interesting only for being modest in size.

Recall the sawtooth function $\psi(x)=x-\lfloor x \rfloor -1/2$. 

\begin{theorem} \label{lemma:kratzel}
Suppose $a<b$ and $h \in C^2[a,b]$ with $h''$ monotonic and nonzero. Then
\[
\Big|\sum_{a<n \leq b} \psi\big(h(n)\big) \Big|\leq 6 \int_a^b |h''(t)|^{1/3} \ud t + 175 \max_{[a,b]} \frac{1}{|h''|^{1/2}}+1 .
\]
\end{theorem}
Kr\"{a}tzel's result has ``$+2$'' as the final term. We obtain the theorem with ``$+1$'' by correcting a gap in the proof and arguing carefully in one case, as explained in \cite[Theorem~A.8]{thesis}. The constant term is irrelevant for our purposes in this paper, in any case.

\section{\bf An elementary lemma}
\begin{lemma} \label{le:squarecompletion}
If $s>0$ and $0<t<1$ then 
\[
s+s^{-1} \leq 2 +t \quad \Longrightarrow \quad | s - 1 | \leq 3\sqrt{t} .
\]
\end{lemma}
\begin{proof} By taking the square root on both sides of the inequality 
\[
( s^{1/2}-s^{-1/2})^2 = s + s^{-1} - 2 \leq t ,
\]
and then using that the number $1$ lies between $s^{1/2}$ and $s^{-1/2}$, we find
\[
| s^{1/2} - 1 | \leq t^{1/2} .
\]
Hence $1 - t^{1/2} \leq s^{1/2} \leq 1 + t^{1/2}$, and now squaring both sides and using that $t<t^{1/2}$ (when $t<1$) proves the lemma. 
\end{proof}

\section*{Acknowledgments}
This research was supported by grants from the Simons Foundation (\#204296 and \#429422 to Richard Laugesen) and the University of Illinois Scholars' Travel Fund, and also by travel funding from the conference \emph{Fifty Years of Hearing Drums: Spectral Geometry and the Legacy of Mark Kac} in Santiago, Chile, May 2016. At that conference, Michiel van den Berg generously discussed  his papers \cite{BBG16a,BBG16b} with Laugesen. The anonymous referee made useful suggestions that improved the organization of the paper. The material in this paper forms part of Shiya Liu's Ph.D. dissertation at the University of Illinois, Urbana--Champaign \cite{thesis}.

\end{document}